\documentclass[11pt]{amsart}
\usepackage{graphicx}
\usepackage{amsmath,amsthm,amssymb}
\usepackage{amsfonts}
\usepackage{indentfirst}

\newtheorem{thm}{Theorem}[section]

\newtheorem{lem}[thm]{Lemma}

\theoremstyle{definition}
\newtheorem{defn}{Definition}[section]

\theoremstyle{remark}

\begin{document}

\title{A fixed point theorem for twist maps}
\author{Peizheng Yu, Zhihong Xia}
\address{Department of Mathematics, Southern University of Science and
  Technology, Shenzhen, China}
\address{Department of Mathematics, Northwestern University, Evanston,
  IL 60208 USA}
\email{11930514@mail.sustech.edu.cn, xia@math.northwestern.edu}
\date{version: June 11, 2021}
\maketitle

\begin{abstract}
  Poincar\'e's last geometric theorem (Poincar\'e-Birkhoff Theorem
  \cite{Birkhoff1913}) states that any area-preserving twist map of
  annulus has at least two fixed points. We replace the
  area-preserving condition with a weaker intersection property, which
  states that any essential simple closed curve intersects its image
  under $f$ at least at one point. The conclusion is that any such map
  has at least one fixed point. Besides providing a new proof to
  Poincar\'e's geometric theorem, our result also has some
  applications to reversible systems.
\end{abstract}

\section{Introduction}

In 1913, Birkhorff proved Poincar\'e's Last Geometric Theorem
\cite{Birkhoff1913} stating that an area-preserving homeomorphism of
an annulus satisfying the twist condition has at least two fixed
points. The proof has been subsequently improved by Barrar
\cite{Barrar1967}, Carter \cite{Carter1982} and many others. In this
note, we replace the area-preserving condition with a weaker
intersection property and obtain a weaker, but sharp, result. We show
that there is at least one fixed point under the intersection
condition. Combining with Slaminka's \cite{Slaminka1993} result on
removing isolated fixed point with zero index, our method provides
another proof of Poincar\'e's last geometric theorem for
area-preserving case. It also provides a new proof to Carter's
theorem~\cite{Carter1982}. One of the interesting applications of our
result is for the so-called reversible systems. It is interesting to
note that both the KAM theory and Aubry-Mather theory are applicable
to reversible systems (Siegel \& Moser \cite{SiegelMoser1971}, Chow \&
Pei \cite{ChowPei1995}). Here we also have a positive answer for
Poincar\'e's geometric theorem, albeit with just one fixed point.

Our proof is more of a simple standard dynamical systems approach,
using techniques of Franks \cite{Franks1988,Franks1992}, and Brouwer's
plane translation theorem (Brown \cite{Brown1984}). We analyse the
chain recurrent set, where  Conley's fundamental theorem for dynamical
systems \cite{Conley1978} provided a fine structure on the
chain recurrent set. The key observation is very simple: the intersection
property implies that two boundary components are in the same
chain transitive component of the chain recurrent set. The techniques
used in the proof are very much standard in dynamics.

More precisely, let $\mathbb{A} = S^1 \times [0,1]$ be
an annulus with boundaries $A_0 = S^1 \times \{0\}$ and
$A_1 = S^1 \times \{1\}$. Let $f:\mathbb{A}\rightarrow\mathbb{A}$ be a
homeomorphism satisfying the twist condition and intersection
property.

We say that $f$ satisfies the {\it intersection property}\/ if any
essential simple closed curve intersects its image under $f$ at least
at one point. A simple closed curve is said to be essential if it is
not contractible. As for the twist condition, we need to consider the
covering space $\widetilde{\mathbb{A}} = \mathbb R \times [0,1]$ of
$\mathbb{A}$, and the lift of $f$ denoted by
$\widetilde{f}:\widetilde{\mathbb{A}}\rightarrow\widetilde{\mathbb{A}}$. We
say that $f$ satisfies the {\it twist condition}\/ provided that
$\widetilde{f}$ moves the two boundary components of
$\widetilde{\mathbb{A}}$ in opposite directions. We may assume that
$\widetilde{f}(x,y) = (x+r_1(x),y)$ for $y = 1$ and
$\widetilde{f}(x,y) = (x-r_0(x),y)$ for $y = 0$, where
$r_0(x), r_1(x) >0$.  Our main result is following theorem.

\begin{thm}
  If $f:\mathbb{A}\rightarrow\mathbb{A}$ is a homeomorphism satisfying
  twist condition and intersection property. Then $f$ has at least one
  fixed point.
\end{thm}

Note that the intersection property is weaker than the area-preserving
property, and Carter gives an example (Figure 1) in \cite{Carter1982} that
such $f$ may have exactly one fixed point. So, our weaker conclusion,
$f$ has at least one fixed point instead of two, is actually sharp.

\begin{figure}[htbp]
	\centering
	\includegraphics{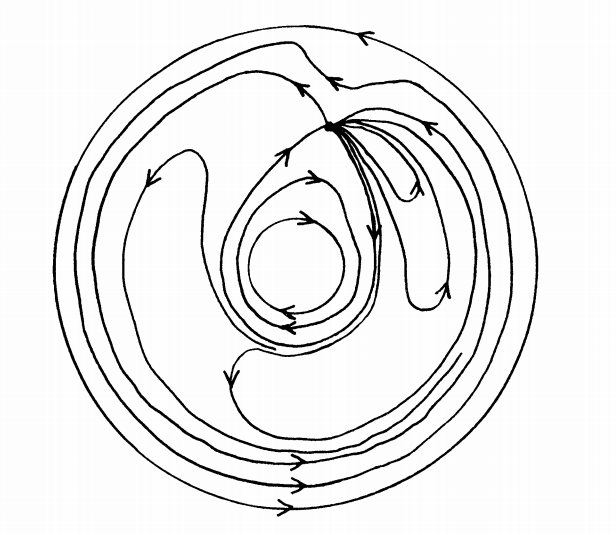}
	\caption{Carter's example}
	\label{}
\end{figure}

There are many different proofs and extensions of Poincar\'e's
geometric theorem. Some proofs are less clear and convincing than
others. Our statement of theorem shares some common features with the
works of Birkhoff \cite{Birkhoff1926} and Carter
\cite{Carter1982}. Carter obtained two fixed points, using a stronger
intersection property.

An outline of this paper is as follows: In section 2, we recall some
properties of chain recurrent set and introduce Conley's fundamental
theorem of dynamical systems and the Brouwer planar translation
theorem. In section 3, we will first assume $f$ is rigid rotation on
boundaries i.e. $r_0(x)$ and $r_1(x)$ are constant and prove our main
theorem, then we will show for the general case, the theorem also
holds true. In section 4, we show how our result applies to
reversible systems.

\section{Chain recurrent set and useful theorems}

Chain recurrence is an important concept in dynamical systems. Given a
compact metric space $X$ and homeomorphism $f: X \rightarrow X$. We
have the following definitions.

\begin{defn}Given two points $p,q \in X$, an {\em $\epsilon$-chain}\/ from
  $p$ to $q$ is a sequence $x_1, x_2, ..., x_n$, where $x_1 = p$,
  $x_n = q$ and for all $1 \leq i \leq n-1$,
  $d(f(x_{i}),x_{i+1}) < \epsilon$.
\end{defn}

\begin{defn} A point $p \in X$ is {\em chain recurrent}\/ if for all
  $\epsilon >0$, there is an $\epsilon$-chain from $p$ to itself. The
  set of all chain recurrent points in X is called the $chain$
  {\em recurrent set}\/, denoted by $\mathcal{R}(f)$.
\end{defn}

Now, we can define an equivalence relation on $\mathcal{R}(f)$ as
follows: For $p,q \in \mathcal{R}(f)$, $p \sim q$ if and only if for
all $\epsilon >0$, there are $\epsilon$-chains from $p$ to $q$ and
from $q$ to $p$. And it's easy to see that $\sim$ is reflexive,
symmetric and transitive.

\begin{defn} The equivalence classes in $\mathcal{R}(f)$ for $\sim$ is
  called $chain$ $transitive$ $components$. A set is called {\em chain
  transitive}\/ if any two points $p$, $q$ in this set, we have
  $p \sim q$.
\end{defn}

The next lemma allows us to relate chain transitive components with
connected components.

\begin{lem} \label{lem2.1}
  The connected components of $\mathcal{R}(f)$ are chain
  transitive.
\end{lem}

\begin{proof}
  Let $K$ be a connected component of $\mathcal{R}(f)$. Given any
  $\epsilon >0$, we can cover $K$ by disks with radius
  $\frac{\epsilon}{4}$. Since $K$ is closed and hence compact, there
  exists a finite subcover $\{B_i\}_{i=1}^m$. Then for any
  $x,y \in K$, we can find a sequence
  $\{x=x_0,x_1,...,x_{n-1},x_n=y\}$ in $\mathcal{R}(f)$ such that
  $x_i$ and $x_{i+1}$ lies in a same disk for $i=0,1,...,n-1$. Note
  each $x_i \in \mathcal{R}(f)$, there is a $\frac{\epsilon}{2}$-chain
  from $x_i$ to $x_i$. We replace the last point by $x_{i+1}$, then we
  get an $\epsilon$-chain form $x_i$ to $x_{i+1}$. And finally we get
  an $\epsilon$-chain form $x$ to $y$. Since $x$ and $y$ are
  arbitrary, $K$ is chain transitive.
\end{proof}

\

In order to state the fundamental theorem of dynamical systems, we
need to give this definition.

\begin{defn} $g: X \rightarrow \mathbb{R}$ is a {\em complete Lyapunov
    function}\/ for $f$ if:

  \begin{enumerate}
  \item $\forall p \notin \mathcal{R}(f)$, $g(f(p)) < g(p)$.
    
 \item  $\forall p,q \in \mathcal{R}(f)$, $g(p) = g(q)$ if and only if
   $p \sim q$.
   
  \item $g(\mathcal{R}(f))$ is compact and nowhere dense in $\mathbb R$.
  \end{enumerate}
\end{defn}

\begin{thm}[Conley's \cite{Conley1978} fundamental theorem of
  dynamical systems] Complete Lyapunov function exists for any
  homeomorphism on compact metric spaces.
\end{thm}

This theorem states that we can get a complete Lyapunov
function that stays constant only on the chain transitive components
and strictly decreases along any orbit not in $\mathcal{R}(f)$. More
details of this theorem can be found in Franks \cite{Franks2017}.

The next concept and theorem are introduced in
\cite{Franks1988}. It is very useful for showing the existence of
a fixed point on plane. And it is easy to see from the definition that
$\epsilon$-chain and disk chain have a closed relationship.

\begin{defn}Let $f:M \rightarrow M$ be a homeomorphism of a surface. A
  {\em disk chain}\/ for $f$ is a finite set $U_1, U_2,...,U_n$ of
  embedded open disks in $M$ satisfying

  \begin{enumerate}
  \item $f\left(U_{i}\right) \cap U_{i} = \varnothing$ for $1 \leq i \leq n$.
  \item If $i \neq j$, then either $U_{i}=U_{j}$ or
  $U_{i} \cap U_{j}=\varnothing$.
  \item For $1 \leq i \leq n$, there exists $m_{i}>0$ with
  $f^{m_{i}}\left(U_{i}\right) \cap U_{i+1} \neq \varnothing$.
    \end{enumerate}

    If $U_1 = U_n$, we will say that $U_1, U_2,...,U_n$ is a $periodic$
  $disk$ $chain$.
\end{defn}

\begin{thm}[Brouwer planar translation theorem] Let
  $f: \mathbb{R}^2 \rightarrow \mathbb{R}^2$ be an orientation
  preserving homeomorphism which possesses a periodic disk chain. Then
  there is a simple closed curve $\gamma$ in $\mathbb{R}^2$ such that
  $I(\gamma,f)=1$. If $f$ has only isolated fixed points, then $f$ has
  a fixed point of positive index.
\end{thm}

The original version of Brouwer's planar translation theorem states
that a fixed point free homeomorphism of $\mathbb{R}^2$ can be
viewed locally as plane translations.

\section{Proof of the Main Theorem}
\label{s:mt}

With the above preparation, we can prove our main theorem.

The rough idea of the proof is as follows: we first consider the case
that $f$ is rigid rotation on the boundaries of the annulus
i.e. $r_0(x)$ and $r_1(x)$ are both constant. The general case can be
embedded in this case and we will explain this in more details in the
end. In the simple case, note that $A_0$ and $A_1$ are both in
$\mathcal{R}(f)$, we will first show, in lemma \ref{lem3.2}, that we
can find an $\epsilon$-chain from $A_0$ to $A_1$ and an
$\epsilon$-chain from $A_1$ to $A_0$. Suppose, on th econtrary, that
$A_0$ and $A_1$ are not in the same chain transitive component, by Conley's
fundamental theorem of dynamical systems, there is a complete Lyapunov
function $g$ defined on $\mathbb{A}$, such that $g(A_0) \neq
g(A_1)$. And this allows us to find an essential closed curve, close
to some level curves of $g$, that doesn't satisfies the intersection
property.

With Lemma \ref{lem3.2} and the twist condition on the boundary, there
is a periodic $\epsilon$-chain in the covering space. Then, similar to
Franks \cite{Franks1988}, we can use this $\epsilon$-chain to construct a
periodic disk chain. And the existence of a fixed point follows from
the Brouwer plane translation theorem.

\

We now proceed with details.

\begin{lem} \label{lem3.2}
  If $f$ is rigid rotation on the boundaries of the annulus, there is
  a chain transitive component $D$ of $\mathcal{R}(f)$ such that
  $D \cap A_0 \neq \varnothing$ and $D \cap A_1 \neq \varnothing$
\end{lem}

\begin{proof}
  We prove by contradiction. First, by Conley's fundamental theorem of
  dynamical system, there is a complet Lyaponuv function $g$ on
  $\mathbb{A}$ with the following three properties:

 \begin{enumerate}
  \item $\forall p \notin \mathcal{R}(f)$, $g(f(p)) < g(p)$.
    
 \item  $\forall p,q \in \mathcal{R}(f)$, $g(p) = g(q)$ if and only if
   $p \sim q$.
   
  \item $g(\mathcal{R}(f))$ is compact and nowhere dense in $\mathbb R$.
  \end{enumerate}
  
  Suppose that, contrary to the conclusion of the lemma, $A_0$ and $A_1$ are not
  in the same chain transitive component, then, by property (2),  the function $g$
  takes different values on $A_0$ and $A_1$. Assume
  $g|_{A_0} = a < g|_{A_1} = b$, where $a$, $b$ are real numbers.
  
 By property (3), there
  is a number $c \in [a, b]$ such that the set
  $C= g^{-1}(c) \subset \mathbb{A}$ does not intersect the chain
  recurrent set of $f$, therefore, by property (1), $g(f(x)) < g(x) =c$, therefore,
  $f(C) \cap C = \varnothing$, and $C$ seperates $A_0$ and $A_1$

 Let $$\epsilon = \inf_{x \in C}  (g(x) - g(f(x)))  = \inf_{x
   \in C} ( c - g(f(x))),$$
 then $\epsilon >0$, by the compactness of $C$.

 The open set $B= g^{-1}((c - \epsilon, c))$ seperates $A_0$ and
 $A_1$. Moreover, $f(B) \cap B = \varnothing$. One can easily choose
 an essential simple closed curve $\gamma$ inside $B$. We have
 $f(\gamma) \cap \gamma = \varnothing$. This contradicts to the
 intersection property.

\end{proof}

One can prove the above lemma directly by using the so-called {\em attractor-repeller
  pair}\/. However, using Lyapunov function is more conceptual.

\
 
 To prove our main theorem, all we need to do is to construct a
 periodic disk chain in the covering space. 

\begin{proof}[Proof of theorem 1.1]
  First, we consider a simple case that $f$ is rigid rotation on the boundaries of the
  annulus. Let's extend $\widetilde{f}: \widetilde{\mathbb{A}} \rightarrow \widetilde{\mathbb{A}}$ to $h: \mathbb{R}^2 \rightarrow \mathbb{R}^2$ as follows:
  \begin{equation*}
  h(x,y)=
  \begin{cases}
  (x+r_1,y), & \text { if } y\geq 1 \\
  \widetilde{f}(x,y)        & \text { if } 0<y<1\\
  (x-r_0,y), & \text { if } y\leq 0 
  \end{cases}
  \end{equation*}

  By lemma \ref{lem3.2}, there exists  $\widetilde{p_0} \in \widetilde{A_0}$,
  $\widetilde{p_1} \in \widetilde{A_1}$ s.t. for any
  $\epsilon > 0$, we have $\epsilon$-chains with respect to $h$ from $\widetilde{p_0}$ to $\widetilde{p_1}$ and from $\widetilde{p_1}+(l_1,0)$ to
  $\widetilde{p_0}+(l_0,0)$ for some  $l_0, l_1 \in \mathbb{Z}$. Because of our boundary assumption, we can choose a suitable $l \in \mathbb{Z}$ such that there exist $\epsilon$-chains from $\widetilde{p_1}$ to $\widetilde{p_1}+(l_1+l,0)$ and from $\widetilde{p_0}+(l_0+l,0)$ to $\widetilde{p_0}$. Note that $h(p+(1,0)) = h(p)+(1,0)$. We obtain an $\epsilon$-chain from $\widetilde{{p_0}}$ to $\widetilde{{p_0}}$ passing through $\widetilde{p_1}$, $\widetilde{p_1}+(l_1+l,0)$ and $\widetilde{p_0}+(l_0+l,0)$.
	
	Prove by contradiction. Suppose  $\widetilde{f}:
        \widetilde{\mathbb{A}} \rightarrow \widetilde{\mathbb{A}}$
        doesn't have fixed points, since $\mathbb{A}$ is compact, it follows that $\exists \delta > 0$ s.t. $\|\widetilde{f}(p)-p\| \geq \delta$ for all $p \in \widetilde{\mathbb{A}}$. Hence, $\|h(p)-p\| \geq \delta$ for all $p \in \mathbb{R}^2$.

	Take $\epsilon < \frac{\delta}{4}$, and assume $\{z_0=\widetilde{{p_0}} ,z_1,...,z_{n-1},z_n=z_0\}$ is the above $\epsilon$-chain from $\widetilde{{p_0}}$ to $\widetilde{{p_0}}$. We will construct a periodic disk chain and apply the Brouwer planar translation theorem to get a contradiction.
	
	Let $U_0 = B_{\epsilon}(z_0), U_1 = B_{\epsilon}(z_1),...,U_{n-1} = B_{\epsilon}(z_{n-1}),U_n = B_{\epsilon}(z_n) = U_0$, where $B_{\epsilon}(z) = \{q:\|q-z\|<\epsilon\}$. Then we have the following properties.
	
        \begin{enumerate}
	\item $\forall 0 \leq i \leq n$, $h(U_i) \cap U_{i+1} \neq \varnothing$.
	\item $\forall 0 \leq i \leq n$, $h(U_i) \cap U_i = \varnothing$.
        \end{enumerate}
	
	
	Now, let's consider two cases. 
	
	One case is that any $i \neq j \in \{0,1,...,n-1\}$ implies
        $U_i \cap U_j = \varnothing$, then $\{U_0,
        U_1,...,U_{n-1},U_n= U_0\}$ is a periodic disk chain and it
        follows from the Brouwer planar translation theorem that $h$
        has a fixed point, which contradicts to our assumption. This
        completes the proof, in this case, when $f$ is rigid rotation on the
        boundaries. 
	
	The other case is that $\exists i_0 \neq j_0 \in
        \{0,1,...,n-1\}$ s.t. $U_{i_0} \cap U_{j_0} \neq
        \varnothing$. Here, we can take $i_0 < j_0$ and consider the
        sequence $\{U_{i_0}, U_{i_0+1},...,U_{j_0-1},U_{j_0}\}$. We
        may assume that this sequence has the property: for any $i
        \neq j \in \{i_0, i_0+1,...,j_0-1,j_0\}$ except for $i=i_0$
        and $j=j_0 $ simultaneously, we have $U_i \cap U_j =
        \varnothing$. Because if not, $\exists i_1 \neq j_1 \in \{i_0,
        i_0+1,...,j_0-1,j_0\}$ s.t. $U_{i_1} \cap U_{j_1} \neq
        \varnothing$ and $\{i_1, i_1+1,...,j_1-1,j_1\}$ is a shorter
        sequence than $\{i_0, i_0+1,...,j_0-1,j_0\}$. Then we can let
        the new $i_0 = i_1$ and new $j_0 = j_1$. Moreover, since
        $U_{i} \cap U_{i+1} = \varnothing$, $\forall i = 0,1,...,n-1$,
        we can always find the shortest sequence satisfying the above
        property and this guarantees the existence of such $i_0$ and
        $j_0$. 
	
	Let $V_0 := U_{i_0} \cup U_{j_0}$, $V_1 := U_{i_0+1}$, ..., $V_{k-1} := U_{j_0-1}$, and $V_k := V_0$. We will show that $\{V_0,V_1,...,V_k\}$ is a periodic disk chain.

	Firstly, since $U_{i_0} \cup U_{j_0} \neq \varnothing$, we have
	\begin{equation*}
	\begin{split}
	d(h(U_{i_0}),U_{j_0}) &:= \inf_{p\in U_{i_0},q\in U_{j_0}} \|h(p)-q\| \geq \inf_{p\in U_{i_0},q\in U_{j_0}} |\|h(p)-p\|-\|q-p\|| \\
	&\geq \delta - \sup_{p\in U_{i_0},q\in U_{j_0}}\|p-q\| \geq \delta - 4\epsilon>0. 
	\end{split}
	\end{equation*}
	and similarly, $d(h(U_{j_0}),U_{i_0})>0. $
	Thus,
	\begin{equation*}
	h(V_0) \cap V_0 = h(U_{i_0} \cup U_{j_0}) \cap (U_{i_0} \cup U_{j_0}) = (h(U_{j_0}) \cap U_{i_0}) \cup (h(U_{i_0}) \cap U_{j_0}) = \varnothing
	\end{equation*}
	Along with the property that $\forall 0 \leq i \leq n$, $h(U_i) \cap U_i = \varnothing$, we have
	$h\left(V_{i}\right) \cap V_{i} = \varnothing$ for $0 \leq i \leq k$.

	Secondly, according to the property that for any $i \neq j \in \{i_0, i_0+1,...,j_0-1,j_0\}$ except for $i=i_0$ and $j=j_0 $ simultaneously we have $U_i \cap U_j = \varnothing$, it is easy to see that if $i \neq j\in \{0, 1,...,k-1\}$, then $V_{i} \cap V_{j}=\varnothing$.

	Thirdly, since $\forall 0 \leq i \leq n$, $h(U_i) \cap U_{i+1} \neq \varnothing$,
	\begin{equation*}
	h(V_0) \cap V_1 = (h(U_{i_0}) \cup h(U_{j_0}) )\cap U_{i_0+1} \supset h(U_{i_0})\cap (U_{i_0+1})\neq \varnothing
	\end{equation*}
	and similarly, $h(V_{k-1}) \cap V_k \neq \varnothing$. Thus,  $\forall 0 \leq i \leq k$, $h\left(V_{i}\right) \cap V_{i+1} \neq \varnothing$.

	Therefore, $h: \mathbb{R}^2 \rightarrow \mathbb{R}^2$
        possesses a periodic disk chain, and hence $h$ has at least
        one fixed point. By the definition of $h$, the fixed point
        must lie in the interior of $\widetilde{\mathbb{A}}$, so this
        point is also fixed by $\widetilde{f}$, a contradiction to our
        assumption. We have thus proved the theorem for the case that
        $f$ is rigid rotation on the boundaries.

        \

We now consider the general case where $f$ on the boundaries are not
rigid rotations. Recall that $\widetilde{f}(x,1) = (x+r_1(x),1)$ and
$\widetilde{f}(x,0) = (x-r_0(x),0)$ with $r_0(x),r_1(x)>0$. Suppose
the rotation numbers for $\widetilde{f}|_{\mathbb{R} \times \{0\}}$
and $\widetilde{f}|_{\mathbb{R} \times \{1\}}$ are $\alpha$ and
$\beta$ respectively. Take a small $\delta >0$, we will extend
$\widetilde{f}$ to $\widetilde{f}_\delta: \mathbb R \times
[-\delta,1+\delta]\rightarrow\mathbb R \times [-\delta,1+\delta]$ in
the following way: 

\begin{equation*}
\widetilde{f}_\delta(x,y)=
\begin{cases}
\frac{1+\delta-y}{\delta}\widetilde{f}(x,1)+\frac{y-1}{\delta}(x+\beta,1+\delta), & \text { if } 1\leq y\leq 1+\delta \\
\widetilde{f}(x,y)        & \text { if } 0<y<1\\
\frac{y+\delta}{\delta}\widetilde{f}(x,0)-\frac{y}{\delta}(x+\alpha,-\delta), & \text { if } -\delta \leq y\leq 0 
\end{cases}
\end{equation*}

We can project $\widetilde{f}_\delta$ through the covering map to the extended
annulus,  we have a map $f_\delta$ that is rigid rotation on both
boundaries of the extended annulus. Let $\mathbb{A}_\delta:= S^1 \times [-\delta,1+\delta] $ with two boundaries $A_{0,\delta}$ and $A_{1,\delta}$; the two extended regions are $D_0:=S^1 \times [-\delta,0]$ and $D_1:=S^1 \times [1,1+\delta]$; the covering space of them are $\widetilde{\mathbb{A}_\delta}$, $\widetilde{A_{0,\delta}}$, $\widetilde{A_{1,\delta}}$,$\widetilde{D_0}$, and $\widetilde{D_1}$ respectively.

Now, We should check that $f_\delta$ still satisfies the intersection property on $\mathbb{A}_\delta$. For any essential simple closed curve $\gamma$, there are three cases.

The first one is the curve $\gamma$ is fully contained in $\mathbb{A}$. Then it must satisfy the intersection property.

The second case is the curve $\gamma$ is fully contained in  $D_0$ or $D_1$. We may assume $\gamma$ is in $D_0$. Consider the covering space $\widetilde{\mathbb{A}_\delta}=\mathbb R \times [-\delta,1+\delta]$, let  $\pi_y :\mathbb R \times [-\delta,1+\delta]\rightarrow[-\delta,1+\delta]$ and $\pi_x :\mathbb R \times [-\delta,1+\delta]\rightarrow\mathbb R$ be the projections. And we can find points $p,q\in\widetilde{\gamma}$ such that $\pi_y(p)=\min{\pi_y(\widetilde{\gamma})}$ and  $\pi_y(q)=\max{\pi_y(\widetilde{\gamma})}$. Let $\widetilde{\gamma}_{p,q}$ be the part of $\widetilde{\gamma}$ connecting $p$ and $q$, then $\widetilde{f}(\widetilde{\gamma}_{p,q})\cap\widetilde{\gamma}\neq \varnothing$, since $ \pi_y(\widetilde{f}(p))  \leq \pi_y(\widetilde{\gamma}|_ {\pi_x(\widetilde{f}(p))})$ and $ \pi_y(\widetilde{f}(q))  \geq \pi_y(\widetilde{\gamma}|_ {\pi_x(\widetilde{f}(q))})$. Therefore, the curve $\gamma$ satisfies the intersection property.

The third case is that part of the curve $\gamma$ is contained in $\mathbb{A}$ denoted by $\gamma_0$ and part of $\gamma$ is contained in $D_0$ or $D_1$ denoted by $\gamma_1$. We first construct $\gamma'$ in the way that we fix $\gamma_0$ in $\mathbb{A}$ and replace $\gamma_1$ with parts of $A_0$ or $A_1$ fixing the endpoints s.t. $\gamma'$ is an essential simple closed curve. Then $f(\gamma')\cap\gamma'\neq \varnothing$. If there is an  intersection point is in $int(\mathbb{A})$, then it is also the intersection point of $f(\gamma)\cap\gamma$. Otherwise, parts of $\gamma'$ in $A_0$ or $A_1$ must intersect its image under $f$. Similar arguement as the second case implies $f(\gamma_1)\cap\gamma_1\neq\varnothing$. So, the curve $\gamma$ still satisfies the intersection property.

In conclusion, we show that $f_\delta$ still satisfies the intersection property on $\mathbb{A}_\delta$. Follow the proof of rigid rotation case, we can show $\widetilde{f_\delta} : \widetilde{\mathbb{A}_\delta}\rightarrow\widetilde{\mathbb{A}_\delta}$ has a fixed point. And it's easy to see $\widetilde{f_\delta}$ has no fixed point on  $D_0$ and $D_1$ since $\widetilde{f_\delta}$ in these regions just moves all points in  $D_0$ or $D_1$ in one direction. Thus, the fixed point we get must lie in $\mathbb{A}$. Hence, it is fixed by $f$.

This completes the proof of our main theorem.
\end{proof}

\section{An application to reversible systems}
We first introduce the definition and some properties of reversible
systems.

\begin{defn}
  A map $R : \mathbb{R}^n\rightarrow\mathbb{R}^n$ is called an {\em
    involution}\/ if it satisfies $R^2 = \text{Id}$.
\end{defn}

An easy example of involution is $R : \mathbb{R}^2\rightarrow\mathbb{R}^2$ s.t. $R(x,y)=(-x,y)$. 

Next, we define general reversible system for both continuous and discrete
dynamical systems on $\mathbb{R}^n$.
\begin{defn}
	A vector field in $\mathbb{R}^n$
	$$\dot{x}=F(t, x) $$
	is said to be {\em reversible}\/ if there exists an $C^1$
        involution $R : \mathbb{R}^n\rightarrow\mathbb{R}^n$ such that
	$$\text{D} R \circ F(-t, R(x))=-F(t, x)$$
	where $\text{D}R$ is the derivative of $R$.
        
	A homeomorphism $f$ is said to be {\em reversible}\/ if there
        is a continuous involution
        $R : \mathbb{R}^n\rightarrow\mathbb{R}^n$ such that
	$$
	f^{-1}=R \circ f \circ R
	$$
\end{defn}

Intuitively, a system is reversible if under some involution it is
transformed to a system which is the same as the original one except
that the time direction is reversed. Reversible system naturally
arises in mechanics (Devaney \cite{Devaney1976}). A simple
non-Hamiltonian example is the system derived following second order
equation
$$\ddot{x} + f(t, x) \dot{x} + g(t, x) =0$$
where $$f(-t, -x) = -f(t, x), \; \; g(-t, -x) = -g(t, x).$$

On the annulus, we take the standard involution for reversible
systems. Let $R: \mathbb{A} \rightarrow \mathbb{A}$ be the involution
that takes $(x, y)$, where
$x\in \mathbb{R}/\mathbb{Z}, \; y \in [0, 1]$ to $(-x, y)$.  A map on
the annulus $f:\mathbb{A}\rightarrow\mathbb{A}$ is said to be
reversible if $f^{-1} = R \circ f \circ R$.

Reversible homeomorphisms may not necessarily have the general
intersection property (cf. Sevryuk \cite{Sevryuk1986}). Given an
essential closed curve $\gamma$, $f(\gamma)$ may not intersect with
$\gamma$. So our theorem does not apply directly. However, the proof
of our theorem is based on analysis of the chain recurrent set
$\mathcal{R}(f)$. It is easy to see that the chain recurrent set is
symmetric with respect $R$ for reversible systems, i.e.,
$R(\mathcal{R}(f)) = \mathcal{R}(f)$, furthermore, the intersection
property holds for any symmetric simple closed curve $\gamma$. More
precisely, if $R(\gamma) = \gamma$, then by
reversibility, $$R(f^{-1}(\gamma)) = f(R(\gamma)) = f(\gamma).$$
Suppose, without loss of generality, $f(\gamma)$ is inside the annulus
bounded by $A_0$ and $\gamma$, then $f^{-1}(\gamma)$ must be outside
of the annulus bounded by $A_0$ and $\gamma$ which contains
$f(\gamma)$, a contradiction.

Our proof works under the intersection property for symmetric simple
closed curves.

We conclude that, by the proof of our theorem, any reversible twist
map on the annulus has at least one fixed point. It turns out that
there must be two fixed points in this case. One can remove index zero
fixed point (Slaminka \cite{Slaminka1993}) by some local modifications
without breaking the intersection property. Carter's \cite{Carter1982}
result does not apply in this case.



\end{document}